\documentclass[11pt]{article}
\usepackage[utf8]{inputenc}
\usepackage{amsmath, amssymb, amsthm}
\usepackage{tabularx}
\usepackage{enumerate,float,subcaption,xcolor,multirow}
\usepackage{subcaption}
\usepackage{hyperref,tikz, tikzpagenodes}
\usetikzlibrary{through,intersections}
\usepackage[sorting=nyt,sortcites=true,maxbibnames=9,autopunct=true,autolang=hyphen,hyperref=true,abbreviate=false, backend=biber]{biblatex}
\usepackage{mathrsfs}
\usepackage{tikz-cd}
\usepackage{comment}
\colorlet{ss}{black!20}
\colorlet{sslabel}{black!50}
\colorlet{backline}{black!10}

\newtheorem{theorem}{Theorem}
\newtheorem{corollary}[theorem]{Corollary}
\newtheorem{lemma}[theorem]{Lemma}

\newtheorem*{remark}{Remark}
\theoremstyle{definition}
\newtheorem{definition}{Definition}

\usepackage[left=30mm,right=30mm,top=20mm,bottom=20mm]{geometry}

\def\0{{\bf 0}}
\def\1{{\bf 1}}
\def\v{{\bf v}}
\def\x{{\bf x}}
\def\y{{\bf y}}
\def\z{{\bf z}}

\def\t{{\bf t}}
\def\q{{\bf q}}
\def\s{{\bf s}}
\def\c{{\bf c}}
\def\o{{\bf 0}}
\def\j{{\bf j}}

\newcommand{\R}{{\mathbb {R}}}

\newcommand{\GM}[1]{$\mathrm{GM}_{#1}$}
\newcommand{\WQH}[1]{$\mathrm{WQH}_{#1}$}
\newcommand{\AH}[1]{$\mathrm{AH}_{#1}$}
\newcommand{\aut}{\mathrm{Aut}}

\newcommand{\erase}[2][says]
\makeatletter
\newcommand{\rref}[2]{\hyperref[#2]{{#1}~\ref*{#2}}}

\usepackage{xcolor}
\usepackage{eurosym}
\title{Counting cospectral graphs obtained via switching}
\author{Aida Abiad\thanks{\texttt{a.abiad.monge@tue.nl}, Department of Mathematics and Computer Science, Eindhoven University of Technology, The Netherlands; Department of Mathematics and Data Science of Vrije Universiteit Brussel, Belgium} \and Nils van de Berg\thanks{\texttt{n.p.v.d.berg@tue.nl}, Department of Mathematics and Computer Science, Eindhoven University of Technology, The Netherlands} \and Robin Simoens\thanks{\texttt{Robin.Simoens@UGent.be},  Department of Mathematics: Analysis, Logic and Discrete Mathematics, Ghent University, Belgium; Department of Mathematics, Universitat Polit\`{e}cnica de Catalunya, Spain}}

\date{}

\begin{document}

\maketitle

\begin{abstract}
Switching is an operation on a graph that does not change the spectrum of the adjacency matrix, thus producing cospectral graphs. An important activity in the field of spectral graph theory is the characterization of graphs by their spectrum. Thus switching provides a tool for disproving the existence of such a characterization.

This paper presents a general framework for counting the number of graphs that have a non-isomorphic cospectral graph through a switching method, expanding on the work by Haemers and Spence [European Journal of Combinatorics, 2004]. Our  framework is based on a different counting approach, which allows it to be used for all known switching methods for the adjacency matrix. From this, we derive asymptotic results, which we complement with computer enumeration results for graphs up to 10 vertices.\\

\noindent \textbf{MCS:} 05C50

\end{abstract}

\section{Introduction}

One of the major open problems in spectral graph theory is Haemers' conjecture (2003), which states that almost all graphs are determined by their adjacency spectrum. More precisely, if $f(n)$ is the fraction of non-isomorphic graphs on $n$ vertices that are uniquely determined by the spectrum of the their adjacency matrix, then the conjecture says that $f(n)\rightarrow 1$ if
$n\rightarrow \infty$. This conjecture plays a special role in the famous graph isomorphism problem, for more details see \cite{which}. Moreover, computational evidence {in favor of the conjecture has been provided in \cite{cospectral12, level2mats}, but it remains far from being
solved. One reason is that it is not clear how to show (without some kind of complete enumeration) that
for a given graph there is no other non-isomorphic graph with the same spectrum (cospectral mate) \cite{Hconference2024}. As a consequence, not many graphs are known to be determined by their spectrum.
It has only been shown successfully for specific families of graphs, the largest such family consisting of specific unicyclic graphs \cite{exponential}. Wang et al. have produced a line of research that culminated in showing that an algebraic condition on the graph does imply that the graph is determined by its spectrum \cite{wangSimple}. 

A lot of attention has also gone to the other side; the construction of cospectral graphs. Two graphs are cospectral if they have the same spectrum and they are called cospectral mates if they are cospectral and non-isomorphic. Schwenk \cite{schwenktrees} made progress in this direction by showing that almost all trees are not determined by their spectrum by providing a construction to obtain cospectral trees. The construction of cospectral graphs does not only have relevance for the Haemers' conjecture, but it is also useful to prove that certain graph properties are not determined by the spectrum, like Hamiltonicity \cite{LIU2020199}, having a perfect matching \cite{blazsik2015cospectral}, the chromatic index for regular graphs \cite{YAN2022} or the zero forcing number \cite{abiadzf}. For general cospectral graphs, the most fruitful construction is the switching method developed by Godsil and McKay \cite{GMswitching} (GM for short). Switching is an operation on a graph that does not change the spectrum of the adjacency matrix. For this operation to work the graph needs a special structure, called a switching set. Since Godsil and McKay's seminal paper, more switching methods for the adjacency matrix have been found, such as Wang-Qiu-Hu (WQH) switching \cite{WQHswitching} and switching methods associated with regular orthogonal matrices of level 2 such as Abiad-Haemers (AH) switching \cite{AHswitching} and its recent extensions \cite{MaoSimilar,switchingpaper}.

In this paper, we exploit the common structure of these switching methods to present an asymptotic expression for counting cospectral graphs obtained via such a switching method. This count is applicable to all switching methods for the adjacency matrix. This expression is then calculated for all known irreducible switching methods up to size $8$. We also make the first comparative study of the frequency that the different switching methods occur in small graphs. Our paper can be considered a sequel to Haemers and Spence's article \cite{enumeration} on the enumeration of cospectral graphs by using Godsil-McKay switching, which in turn is also a sequel to the work of Godsil and McKay \cite{GMswitching}.

In \cite{enumeration} the authors obtained a lower bound on the number of graphs on $n$ vertices that have a cospectral mate. Let $g_n$ be the number of graphs on $n$ vertices up to isomorphism.

\begin{theorem}[{\cite[Theorem~3]{enumeration}}]\label{thm:gmcount}
     There are at least $n^3g_{n-1}(\frac1{24}-o(1))$ graphs on $n$ vertices with a cospectral mate.
\end{theorem}

The bound in Theorem~\ref{thm:gmcount} was derived by counting the number of cospectral mates by GM-switching on a switching set of size $4$. Therefore, it is also a lower bound on the number of graphs which have a cospectral mate via \GM4-switching.  

Using a different approach we prove a general method to count the number of cospectral graphs that are obtained by a switching method for the adjacency matrix, see Theorem \ref{thm:generalasymp}. The result contains some parameters of the orthogonal matrix $Q$ related to the switching method and $\Gamma$, the induced subgraph on the switching set. It also uses some new terminology which is introduced in Section~\ref{sec:prelimswitching}. 

\begin{theorem}\label{thm:generalasymp}
    If $Q$ is an $m\times m$ orthogonal matrix such that $(Q,\Gamma)$-switching is distinguishing and produces at least one cospectral mate, then there are 
   \begin{equation}
       \frac{1}{|\aut_Q(\Gamma)|}|\mathcal{V}_Q|^{n-m}g_{n-m}(1+o(1))
   \end{equation}
    graphs on $n$ vertices that have a cospectral mate through $(Q,\Gamma)$-switching.
\end{theorem}

It should be noted that $o(1)$ is related to $n$ going to infinity, but $m$ should remain constant and $o(1)$ depends on this constant.

Theorem~\ref{thm:generalasymp} generalizes Theorem~\ref{thm:gmcount} and gives an asymptotically matching upper bound. Our proof takes a different approach than the one by Haemers and Spence, because their proof relies on the specifics of \GM4-switching, which do not extend to all other switching methods.  Theorem~\ref{thm:generalasymp} applies to every switching method that is \emph{distinguishing}, see Definition~\ref{def:distinguishing}, which is a condition that is satisfied by all currently known switching methods for the adjacency matrix (\GM4, \GM6, \GM8, \WQH6, \WQH8, \AH6 and Fano switching).
Concerning the enumeration of specific switching methods, Theorem~\ref{thm:generalasymp} gives an asymptotically tight bound.

This paper is structured as follows. In Section~\ref{sec:prelim}, we introduce the definitions needed to understand the statement and proof of Theorem~\ref{thm:generalasymp}. We also provide an overview of results and switching methods from the literature. Theorem~\ref{thm:generalasymp} is proven in Section~\ref{sec:thmproof}, and it is applied to several switching methods in Section~\ref{sec:consequences}. In Section~\ref{sec:computer}, we provide computer results for switching in small graphs, complementing the asymptotic results for large graphs that follow from Theorem~\ref{thm:generalasymp}. Finally, in Section~\ref{sec:conclusion}, we discuss the possibility of generalizing our main result and its use for obtaining new insights towards Haemers' conjecture.

\section{Preliminaries}\label{sec:prelim}

In this paper, a graph $G = (V,E)$ is considered simple and loopless, with vertex set $V = [n] := \{1,\dots,n\}$. The graphs are used interchangeably with their \emph{adjacency matrix}, the symmetric matrix $A = (a_{ij})$ where $a_{ij} = 1$ if and only if there is an edge joining $i$ and $j$, and $0$ otherwise. The \emph{spectrum} of a graph is the multiset of eigenvalues of its adjacency matrix. Graphs are \emph{cospectral} if they have the same spectrum. If two  cospectral graphs are non-isomorphic, they are \emph{ cospectral mates}. A graph is determined by its spectrum if it has no cospectral mates. In this paper, we study graphs  that are not determined by their spectrum by enumerating the graphs that have a cospectral mate.

If two graphs are cospectral, then there is an orthogonal matrix $Q$ such that their adjacency matrices $A$ and $A'$ satisfy $Q^TAQ = A'$. If this matrix $Q$ is an (0,1)-matrix, then it is a permutation matrix and the graphs are isomorphic. An orthogonal matrix is \emph{regular} if it has constant row sum. We say that two matrices $U,V$ are \emph{equivalent} if there are permutation matrices $P,R$ such that $U = P^TVR$. Equivalent orthogonal matrices produce the same pair of cospectral graphs up to isomorphism. The \emph{level} of an orthogonal matrix Q is the smallest positive integer \(\ell\) such that \(\ell Q\) is an integral matrix, or \(\infty\) if $Q$ has irrational entries. A matrix is \emph{decomposable} if it is equivalent to a non-trivial block-diagonal matrix; otherwise, it is \emph{indecomposable}. For example, the adjacency matrix of a disconnected graph is decomposable. The matrices in this paper are often described as a block matrix consisting of smaller matrices. These smaller $n\times n$ matrices consist of the identity matrix $I_n$, the all-zeros matrix $O_n$, the all-ones matrix $J_n$ and $Y_n=nI_n-J_n$. The columns of $J_n$ are the all-ones vector $\j_n$. If the size is clear, we omit the subscript. 

If $A+rJ$ and $A'+rJ$ have the same spectrum for all real numbers $r$, then $A$ and $A'$ (and the associated graphs) are \emph{$\R$-cospectral}. Johnson and Newman \cite{rcospectral} showed that $A$ and $A'$ are $\R$-cospectral if and only if there is a regular rational orthogonal $Q$ such that $Q^TAQ = A'$. Furthermore, they showed that this is equivalent to the graphs being \emph{generalized cospectral}, that is, being cospectral and having cospectral complements.

\subsection{Enumeration}

 The enumeration of graphs with a cospectral mate relies on the enumeration of graphs up to isomorphism. A graph is \emph{asymmetric} if it has no non-trivial automorphism. We use the following known results:

\begin{theorem}[{\cite[Theorem~2]{asymmetricgraphs}}]\label{thm:asymmetric}
    Almost all graphs are asymmetric. This means that the proportion of asymmetric graphs on $n$ vertices goes to $1$ as $n \to \infty$.
\end{theorem}

\begin{corollary}[\cite{oberschelp67}]\label{cor:nonisomgraphs}
    There are $\frac{1}{n!}2^{\binom{n}{2}}(1 + o(1))$ non-isomorphic graphs on $n$ vertices.
\end{corollary}
Let \(g_n\) be the number of 
graphs on \(n\) (unlabelled) vertices, that is, \(g_n=\frac{1}{n!}2^{\binom{n}{2}}(1 + o(1))\). 

 The approach in \cite{enumeration} for enumerating graphs with a cospectral mate can be generalized to all switching methods for the adjacency matrix. 

\subsection{Overview of switching methods}\label{sec:switching}
In this section, we outline the different switching methods that are considered in this paper. We focus on switching methods for the adjacency matrix (\cite{GMswitching, WQHswitching, AHswitching, MaoSimilar, switchingpaper}), and each of such methods is associated with a regular orthogonal matrix, so by \cite{rcospectral} we know that they create graphs that are $\mathbb{R}$-cospectral (generalized cospectral). Hence the same switching method can always be applied to the complement of a graph. Moreover, all the switching methods we consider are associated with an indecomposable orthogonal matrix.

GM- and WQH-switching have found applications for finding new strongly regular graphs \cite{ihringer2019new, barwick} and for showing that distance regularity is not determined by the spectrum \cite{VANDAM20061805}.

\begin{theorem}[Godsil-McKay switching \cite{GMswitching}]
    In a graph $G = (V,E)$, a set of $2k$ vertices $X$ forms a GM-switching set if
    \begin{itemize}
        \item the induced subgraph $G[X]$ is regular,
        \item each vertex $v$ in $V\setminus X$ has $0,k$ or $2k$ neighbors in $X$.
    \end{itemize}
    For every vertex $v$ outside the switching set with $k$ neighbors in $X$ and every vertex $x$ in $X$ change the adjacency. This creates a new graph $G'$ which is cospectral with $G$.
\end{theorem}

A version of GM-switching can also be applied to the Laplacian matrix, but in this paper we only focus on the adjacency matrix.

\begin{theorem}[Wang-Qiu-Hu switching \cite{WQHswitching}]
    In a graph $G = (V,E)$ two disjoint sets $C_0,C_1$ of $k$ vertices form a WQH-switching set if
    \begin{itemize}
        \item For each $i \in \{0,1\}$ and vertex $v \in C_i$ the difference $\deg(v,C_{1-i}) - \deg(v, C_i)$ has the same value. 
        \item Any vertex $v$ in $V\setminus (C_0\cup C_1)$ has the same number of neighbors modulo $k$ in $C_0$ and $C_1$.
    \end{itemize}
    For any vertex $v$ outside the switching set with $\{\deg(v,C_0), \deg(v,C_1)\} = \{0,k\}$ switch the adjacencies between $v$ and $C_0\cup C_1$. This creates a new graph $G'$ which is cospectral with $G$.
\end{theorem}

If the switching set, $X$ or $C_0 \cup C_1$, is of size $2k$, the switching is called a \GM{2k} or \WQH{2k} switching respectively. The orthogonal matrices associated  with \GM{2k} and \WQH{2k} switching are the level $k$ matrices

\begin{equation*}
    Q_{\text{\GM{2k}}} = \frac{1}{k}\begin{pmatrix} -Y_k & J_k \\ J_k & -Y_k \end{pmatrix} \quad \text{ and }\quad Q_{\text{\WQH{2k}}} =  \frac{1}{k}\begin{pmatrix} Y_k & J_k \\ J_k & Y_k \end{pmatrix}. 
\end{equation*}

Note that \WQH{2k} switching coincides with \GM{2k} switching for $k \leq 2$, because then the orthogonal matrices are equivalent. For $k \geq 3$ the corresponding orthogonal matrices are not equivalent and they are different switching methods. \GM2-switching has level $1$ and is the same as a permutation of the two vertices in the switching set, so \GM4-switching (or equivalently \WQH4-switching) is the smallest of these switching methods that can produce non-isomorphic graphs. Both these switching methods have also been proposed in a more general form \cite{GMswitching, QIU2020265}, but the description stated here covers the case where the orthogonal matrix has one indecomposable block. The other known small switching methods were first described in \cite{AHswitching}. These switching methods were given a combinatorial description in \cite{switchingpaper}. In the paper \cite{switchingpaper}, we analyzed which of these switching methods reduced to smaller switching methods. We call a method \emph{irreducible} if it does not reduce to smaller switching methods, see Definition~\ref{def:reducibility}. Any graph that has a cospectral mate through a reducible switching already has one through one of the smaller switching methods. This is why we only consider the irreducible switching methods. In particular, we only use \AH6 and Fano switching to refer to the irreducible cases.

\begin{theorem}[\AH6-switching, \cite{switchingpaper,AHswitching}]\label{thm:ahswitching}
    Let $G$ be a graph with a triple $(C_1,C_2,C_3)$ of disjoint pairs of vertices such that  
    \begin{enumerate}[i)]
        \item the induced subgraph on $C_1\cup C_2 \cup C_3$ is
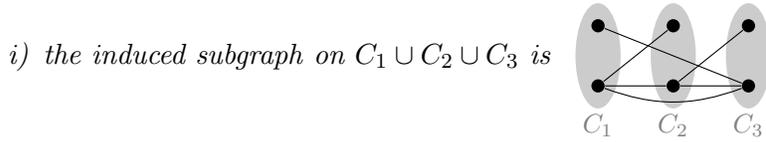
    \begin{tikzpicture}[baseline=(O.base)]
    \path (0,-.5) node (O) {};
    \path (-1.5,-.1) rectangle (0,0);
    \fill[ss] (-1,-.4) ellipse (.3 and .7)
    (0,-.4) ellipse (.3 and .7)
    (1,-.4) ellipse (.3 and .7);
    \path[every node/.append style={circle, fill=black, minimum size=5pt, label distance=2pt, inner sep=0pt}]
    (-1,0) node (1) {}
    (-1,-.8) node[label={[label distance=5pt]270:\color{sslabel}\small\(C_1\)}] (2) {}
    (0,0) node (3) {}
    (0,-.8) node[label={[label distance=5pt]270:\color{sslabel}\small\(C_2\)}] (4) {}
    (1,0) node (5) {}
    (1,-.8) node[label={[label distance=5pt]270:\color{sslabel}\small\(C_3\)}] (6) {};
    \draw (2) edge (4) edge (3) edge[out=-20,in=200] (6)
    (4) edge (5) edge (6)
    (6) edge (1);
\end{tikzpicture}
    \item any vertex $v$ outside $C_1\cup C_2 \cup C_3$ satisfies $\deg(v, C_1) \equiv \deg(v,C_2) \equiv \deg(v,C_3) \pmod{2}$.
    \end{enumerate}
     Let \(\pi\) be the permutation on \(C_1\cup C_2\cup C_3\) that shifts the vertices cyclically to the left. For every \(v\in D\) that has exactly one neighbor \(w\) in each \(C_i\), replace each edge \(\{v,w\}\) by \(\{v,\pi(w)\}\).
    Replace the induced subgraph on \(C_1\cup C_2\cup C_3\) by:
    
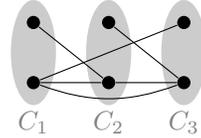
\[\begin{tikzpicture}[baseline=(O.base)]
    \path (0,-.5) node (O) {};
    \path (-1.5,-.1) rectangle (0,0);
    \fill[ss] (-1,-.4) ellipse (.3 and .7)
    (0,-.4) ellipse (.3 and .7)
    (1,-.4) ellipse (.3 and .7);
    \path[every node/.append style={circle, fill=black, minimum size=5pt, label distance=2pt, inner sep=0pt}]
    (-1,0) node (1) {}
    (-1,-.8) node[label={[label distance=5pt]270:\color{sslabel}\small\(C_1\)}] (2) {}
    (0,0) node (3) {}
    (0,-.8) node[label={[label distance=5pt]270:\color{sslabel}\small\(C_2\)}] (4) {}
    (1,0) node (5) {}
    (1,-.8) node[label={[label distance=5pt]270:\color{sslabel}\small\(C_3\)}] (6) {};
    \draw (2) edge (4) edge (5) edge[out=-20,in=200] (6)
    (4) edge (1) edge (6)
    (6) edge (3);
\end{tikzpicture}\]
    The resulting graph is \(\mathbb{R}\)-cospectral with \(G\).
\end{theorem}


\begin{theorem}[Fano switching, \cite{AHswitching,switchingpaper}]\label{thm:fano}
    Let \(G\) be a graph and \(C=\{v_1,\dots,v_7\}\) a subset of its vertices. Let \(\pi\) be the cyclic permutation \((v_i\mapsto v_{i+1\pmod{7}})\). Define \(\ell:=\{v_1,v_2,v_4\}\), \(\mathcal{O}:=\{v_3,v_5,v_6\}\), and \(\ell_i=\pi^i(\ell)\) and \(\mathcal{O}_i=\pi^i(\mathcal{O})\) for all \(i\in\mathbb{Z}/7\mathbb{Z}\). The set \(C\), together with the ``lines'' \(\ell_i\), form the Fano plane \(\mathrm{PG}(2,2)\).
    Suppose that: 
    \begin{enumerate}[(i)]
        \item The induced subgraph on \(C\) is (a) \(\newcommand{\radius}{1.2}
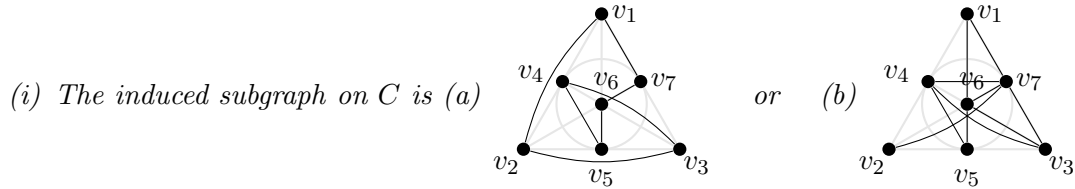
\begin{tikzpicture}[baseline=(O.base)]
    \path (0,0) node (O) {};
    \draw[thick,backline] (-30:\radius) -- coordinate (P1)
    (90:\radius) -- coordinate (P2)
    (210:\radius) -- coordinate (P3) cycle;
    \draw[thick,backline] (210:\radius) -- (P1) (-30:\radius) -- (P2) (90:\radius) -- (P3);
    \node[draw,thick,backline] at (O) [circle through=(P1)] {};
    \path[every node/.append style={circle, fill=black, minimum size=5pt, label distance=0pt, inner sep=0pt}]
    (O) node[label={[xshift=2pt]90:\(v_6\)}] (6) {}
    (P1) node[label={0:\(v_7\)}] (7) {}
    (P2) node[label={[label distance=3pt,yshift=3pt]180:\(v_4\)}] (4) {}
    (P3) node[label={[label distance=2pt]270:\(v_5\)}] (5) {}
    (-30:\radius) node[label={-40:\(v_3\)}] (3) {}
    (90:\radius) node[label={[label distance=0pt]0:\(v_1\)}] (1) {}
    (210:\radius) node[label={[label distance=0pt]220:\(v_2\)}] (2) {};
    \draw (1) edge[out=-135,in=75] (2)
    (2) edge[out=-15,in=195] (3)
    (3) edge[out=135,in=-15] (4)
    (4) edge (5)
    (5) edge (6)
    (6) edge (7)
    (7) edge (1);
\end{tikzpicture}\quad\text{ or }\quad\text{(b)}
\begin{tikzpicture}[baseline=(O.base)]
    \path (0,0) node (O) {};
    \draw[thick,backline] (-30:\radius) -- coordinate (P1)
    (90:\radius) -- coordinate (P2)
    (210:\radius) -- coordinate (P3) cycle;
    \draw[thick,backline] (210:\radius) -- (P1) (-30:\radius) -- (P2) (90:\radius) -- (P3);
    \node[draw,thick,backline] at (O) [circle through=(P1)] {};
    \path[every node/.append style={circle, fill=black, minimum size=5pt, label distance=0pt, inner sep=0pt}]
    (O) node[label={[xshift=2pt]90:\(v_6\)}] (6) {}
    (P1) node[label={0:\(v_7\)}] (7) {}
    (P2) node[label={[label distance=3pt,yshift=3pt]180:\(v_4\)}] (4) {}
    (P3) node[label={[label distance=2pt]270:\(v_5\)}] (5) {}
    (-30:\radius) node[label={-40:\(v_3\)}] (3) {}
    (90:\radius) node[label={[label distance=0pt]0:\(v_1\)}] (1) {}
    (210:\radius) node[label={[label distance=0pt]220:\(v_2\)}] (2) {};
    \draw (4) edge[out=-45,in=165] (3) edge (5)
    (6) edge (1) edge (3) edge (5)
    (7) edge (1) edge[out=225,in=15] (2) edge (3) edge (4) edge (6);
\end{tikzpicture}\)
        \item Every vertex outside \(C\) is adjacent to either:
        \begin{enumerate}[1.]
            \item All vertices of \(C\).
            \item No vertices of \(C\).
            \item Three vertices of \(C\) contained in a line.
            \item Four vertices of \(C\) not contained in a line.
        \end{enumerate}
    \end{enumerate}
    For every \(v\) outside \(C\) that is (non)adjacent to the three vertices of a line \(\ell_i\), make it (non)adjacent to the three vertices of \(\mathcal{O}_i\). In case (b) above, replace the induced subgraph on \(C\) by: \[\newcommand{\radius}{1.2}
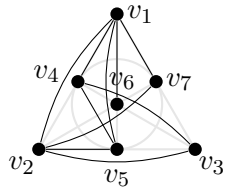
\begin{tikzpicture}[baseline=(O.base)]
    \path (0,0) node (O) {};
    \draw[thick,backline] (-30:\radius) -- coordinate (P1)
    (90:\radius) -- coordinate (P2)
    (210:\radius) -- coordinate (P3) cycle;
    \draw[thick,backline] (210:\radius) -- (P1) (-30:\radius) -- (P2) (90:\radius) -- (P3);
    \node[draw,thick,backline] at (O) [circle through=(P1)] {};
    \path[every node/.append style={circle, fill=black, minimum size=5pt, label distance=0pt, inner sep=0pt}]
    (O) node[label={[xshift=2pt]90:\(v_6\)}] (6) {}
    (P1) node[label={0:\(v_7\)}] (7) {}
    (P2) node[label={[label distance=3pt,yshift=3pt]180:\(v_4\)}] (4) {}
    (P3) node[label={[label distance=2pt]270:\(v_5\)}] (5) {}
    (-30:\radius) node[label={-40:\(v_3\)}] (3) {}
    (90:\radius) node[label={[label distance=0pt]0:\(v_1\)}] (1) {}
    (210:\radius) node[label={[label distance=0pt]220:\(v_2\)}] (2) {};
    \draw (4) edge[out=-15,in=135] (3) edge (5)
    (2) edge[out=-15,in=195] (3) edge (5) edge[out=15,in=-135] (7)
    (1) edge[out=-135,in=75] (2) edge (4) edge[out=-105,in=105] (5) edge (6) edge (7);
\end{tikzpicture}\] The resulting graph is \(\mathbb{R}\)-cospectral with \(G\).
\end{theorem}

The associated regular orthogonal matrices of level $2$ are called $Q_{AH_6}$ and $Q_{Fano}$ and are as follows:
\[
Q_{AH_6} = \frac{1}{2}\left[
\arraycolsep=2.5pt\def\arraystretch{.9}\begin{array}{rrrrrr}
1 & 1     & 0 & 0        & 1 & -1     \\
1 & 1     & 0 & 0        & -1 & 1     \\
1 & -1       & 1 & 1      & 0 & 0       \\
-1 & 1       & 1 & 1      & 0 & 0       \\
0 & 0       & 1 & -1      & 1 & 1    \\
0 & 0       & -1 & 1      & 1 & 1    
\end{array}
\right]\text{ and }  Q_{Fano} = \frac{1}{2}\left[
\arraycolsep=4pt\def\arraystretch{.9}\begin{array}{rrrrrrr}
-1 & 1 & 1 & 0 & 1 & 0 & 0 \\
0 & -1 & 1 & 1 & 0 & 1 & 0 \\
0 & 0 & -1 & 1 & 1 & 0 & 1 \\
1 & 0 & 0 & -1 & 1 & 1 & 0 \\
0 & 1 & 0 & 0 & -1 & 1 & 1 \\
1 & 0 & 1 & 0 & 0 & -1 & 1 \\
1 & 1 & 0 & 1 & 0 & 0 & -1
\end{array}
\right].
\]

 It was shown in \cite{switchingpaper} that \GM4, \AH6 and Fano switching are the only irreducible switching methods of level 2 that use orthogonal matrices of dimension at most 8.

\subsection{Switching}\label{sec:prelimswitching}

In this subsection we give a definition of `a switching method'  that unifies the methods in the previous section. Our framework, which works for all adjacency switching methods, needs a version of this general definition that we introduce in Definition~\ref{def:qgamma}. We also introduce the other terminology needed to understand Theorem~\ref{thm:generalasymp}.
 
\begin{definition}\label{def:qswitching}
    Let $Q$ be a real orthogonal matrix with no integral columns and let \(G\) be a graph with adjacency matrix $A$. A \emph{$Q$-switching} on \(G\) is an operation which returns the graph $G'$ that has adjacency matrix $(Q \oplus I)^TA(Q\oplus I)$, where $I$ is an identity matrix of the appropriate dimension. The rows and columns of $Q \oplus I$ can be labelled by the vertices of $G$. The vertices of $G$ that correspond to the rows and columns of $Q$, form the \emph{switching set}. The induced subgraph on the switching set is the \emph{$Q$-switching graph} and is often denoted by $\Gamma$. 
\end{definition}

An integral column in a real orthogonal matrix is a column with exactly one $1$. A matrix with such columns is equivalent up to permutation of the rows and columns to a matrix of the form $Q \oplus I$. Hence we exclude the matrix $Q$ from having integral columns.

Given a $Q$-switching on $G$, we can partition the vertices of $G$ into the switching set and its complement. Let $$A = \begin{pmatrix} B& V \\ V^T & C \end{pmatrix} $$ be the associated block form of the adjacency matrix, where $B$ is the adjacency matrix of the $Q$-switching graph. Now the graph obtained by switching has adjacency matrix 
\[ 
        A' = \begin{pmatrix} Q^T & 0 \\ 0 & I \end{pmatrix}A\begin{pmatrix} Q & 0 \\ 0 & I\end{pmatrix} =  \begin{pmatrix}
            Q^TBQ & Q^TV \\ V^TQ & C
        \end{pmatrix}.
\]

Therefore, one can list all possible $Q$-switching graphs by testing whether $Q^TBQ$ is the adjacency matrix of a graph and checking that $Q^TV$ (and thus also $V^TQ$) are (0,1)-matrices. 
\begin{definition}
    The set of all $Q$-switching graphs is denoted by $\mathcal{B}_Q$. Given an $m\times m$ orthogonal matrix $Q$, we call a vector $\v \in \{0,1\}^m$ a \emph{$Q$-respecting vector} if $Q^T\v \in \{0,1\}^m$. The set of all $Q$-respecting vectors is denoted by $\mathcal{V}_Q$.
\end{definition} 
Note that all the columns of $V$ must be $Q$-respecting vectors. Our main result, Theorem~\ref{thm:generalasymp}, performs the enumeration of cospectral mates generated by each $Q$ per $Q$-switching graph. The following definition makes this distinction precise.

\begin{definition}\label{def:qgamma}
    For an $m\times m$ real orthogonal matrix $Q$ and a labelled graph $\Gamma$ with vertex set $[m]$, a \emph{$(Q,\Gamma)$-switching} of a graph $G$ is a $Q$-switching of $G$ which can be represented by $(Q \oplus I)^TA(G)(Q \oplus I)$, where the ordering of the vertices starts with a tuple $X = (x_1,\dots, x_m)$ such that $A(X) = A(\Gamma)$. The tuple $X$ is the \emph{$(Q,\Gamma)$-switching tuple}.
\end{definition}

 A $(Q, \Gamma)$-switching only exists if $\Gamma$ is a $Q$-switching graph of $Q$. With this definition we can now be more precise about irreducibility.
\begin{definition}\label{def:reducibility}
    Let $Q$ be an indecomposable orthogonal matrix. A $(Q,\Gamma)$-switching method is \emph{reducible} if it can be obtained by a sequence of \emph{smaller switching methods}, that is, switching methods coming from matrices whose largest indecomposable block is smaller than the size of \(Q\).
    Otherwise, it is \emph{irreducible}.

\end{definition}

In terms of enumeration, we have to distinguish between different $Q$-switching graphs $\Gamma$, because the $Q$-switching graphs have different amounts of symmetry. 

\begin{definition}
    Consider a $(Q,\Gamma)$-switching. The group $\aut(\Gamma)$ acts on $\{0,1\}^m$ by permuting the entries according to the permutation of the vertices in $\Gamma$. Let \emph{$\aut_Q(\Gamma)$} be the subgroup that sends $\mathcal{V}_Q$ to itself.
\end{definition}

We note here that $\aut_Q(\Gamma)$ contains at least all elements of $\aut(\Gamma)$ that fix the matrix $Q$. Associated to each permutation $p \in \aut(\Gamma)$ is a matrix $P$. The action of $p$ on a matrix $Q$ sends the matrix to $P^TQP$. If this action fixes $Q$, that is, $P^TQP = Q$, then for any $\v \in \mathcal{V}_Q$ it holds that $QP^T\v = P^TQ^T\v$ is a (0,1)-vector, so $P^T\v \in \mathcal{V}_Q$. Hence any permutation $p \in \aut(\Gamma)$ that fixes $Q$ lies in $\aut_Q(\Gamma)$.

In a large graph with a $(Q,\Gamma)$-switching the automorphisms of $\aut_Q(\Gamma)$ might extend to isomorphisms of the whole graph. In the proof of Theorem~\ref{thm:generalasymp}, we introduce conditions to control the isomorphism classes of the graphs we consider. In particular, we want the graphs we construct to be asymmetric. If there is a non-trivial permutation in $\aut_Q(\Gamma)$ that sends each $Q$-respecting vector to itself, then this would always be an automorphism of the whole graph $G$. We also want to make sure there is no permutation that changes the graph exactly like the switching method, because then the switching method would never create non-isomorphic graphs.

\begin{definition}\label{def:distinguishing}
    A $(Q,\Gamma)$-switching method is \emph{distinguishing} if $\aut_Q(\Gamma)$ acts faithfully on $\mathcal{V}_Q$. It is \emph{switching-distinguishing} if there is no permutation $\phi : [m] \to [m]$ such that $\phi \cdot A(\Gamma) = Q^TA(\Gamma)Q$ and such that $\v_i = (Q^T\v)_{\phi(i)}$  for all $Q$-respecting vectors $\v$ and all $1\leq i \leq m$.
\end{definition}

In Section \ref{sec:consequences} we show that GM, WQH and some other switching methods have matrices for which any switching method is distinguishing and switching-distinguishing. One can show this in one go by simply showing that there is no permutation that induces the appropriate fixing of the $Q$-respecting vectors. In fact any proper switching method is switching-distinguishing.

\begin{lemma}
    If a $(Q,\Gamma)$-switching method produces a pair of cospectral mates then it is switching-distinguishing.
\end{lemma}

\begin{proof}
    By contraposition. Let $G$ be a graph with a $(Q,\Gamma)$-switching. Any permutation $\phi \in \aut_Q(\Gamma)$ that satisfies $\v_i = (Q^T\v)_{\phi(i)}$ for all $Q$-respecting vectors can be extended to a permutation on the vertices of $G$ by acting as the identity outside the switching set. This permutation, by definition of the switching, is an isomorphism between $G$ and the switched graph. This holds for any graph $G$, so $(Q,\Gamma)$-switching does not produce cospectral mates.
\end{proof}

It is not known to the authors which $(Q,\Gamma)$-switching methods are distinguishing and hence this property is needed as an assumption in Theorem~\ref{thm:generalasymp}.

\section{Asymptotic number of cospectral graphs per switching method}\label{sec:thmproof}

In this section we prove Theorem~\ref{thm:generalasymp}.

 \newtheorem*{thm:main}{Theorem \ref{thm:generalasymp}}
\begin{thm:main}
    If $Q$ is an $m\times m$ orthogonal matrix such that $(Q,\Gamma)$-switching is distinguishing and produces at least one cospectral mate, then there are 
   \begin{equation}\label{eq:asympformula}
       \frac{1}{|\aut_Q(\Gamma)|}|\mathcal{V}_Q|^{n-m}g_{n-m}(1+o(1))
   \end{equation}
    graphs on $n$ vertices that have a cospectral mate through $(Q,\Gamma)$-switching.
\end{thm:main}

The proof consists of the upper bound, which is proved in Lemma~\ref{lem:upperbound} and the lower bound, which is proved in Lemma~\ref{lem:lowerbound}. 
\subsection{Proof setup}

Theorem~\ref{thm:generalasymp} counts graphs up to isomorphism, but this is equivalent to counting adjacency matrices up to conjugation by permutation matrices. Although in the proof we use the matrix formulation, we use the word isomorphism to mean graph isomorphism between the adjacency matrices.

We look at the set $\mathcal{T}_n$ of all vertex-labelled graphs with vertex set $[n]$ for which $(1,\dots,m)$ is a $(Q,\Gamma)$-switching tuple. To be precise, for each isomorphism class of graphs on $k$ vertices, we choose one representative. Let $\mathcal{A}_k$ 
be the set of adjacency matrices of these representatives. If $A(\Gamma)$ denotes the adjacency matrix of the labelled graph $\Gamma$, then each graph of order $n$ with a $(Q,\Gamma)$-switching is isomorphic to a graph whose adjacency matrix is in the set
\[
 \mathcal{T}_n(Q,\Gamma) := \left\{ \begin{pmatrix} A(\Gamma) & V \\ V^T & C\end{pmatrix} : C \in \mathcal{A}_{n-m} \text{ and columns of } V \text{ are in } \mathcal{V}_Q\right\}.
\]

We want to focus on the subset of those graphs where: 
\begin{enumerate}
    \item The set of vertices outside the switching set induce an asymmetric graph.
    \item All of the $Q$-respecting vectors appear as a column in the top-right block of the adjacency matrix.
\end{enumerate}

In terms of matrices this is defined as follows.
\begin{definition}
    Let $\mathcal{S}_n(Q,\Gamma) \subseteq \mathcal{T}_n(Q,\Gamma)$ be the subset of matrices where each $Q$-respecting vector appears as a column in $V$ and where $C$ is the adjacency matrix of an asymmetric graph.
\end{definition}

In the rest of this section, we fix an orthogonal matrix $Q$ and a $Q$-switching graph $\Gamma$ and write $ \mathcal{S}_n, \mathcal{T}_n$ for $\mathcal{S}_n(Q,\Gamma)$ and $\mathcal{T}_n(Q,\Gamma)$. The following observation is needed for both the upper and lower bound.

\begin{lemma}\label{lem:sizeS}
For any order $n > m$,
    \[
    |\mathcal{T}_n| = |\mathcal{V}_Q|^{n-m}g_{n-m}
    \]
    and as $n \to \infty$,
    $$|\mathcal{S}_n| = |\mathcal{V}_Q|^{n-m}g_{n-m}(1-o(1)).$$
\end{lemma}

\begin{proof}
    By definition $\mathcal{A}_{n-m}$ has size $g_{n-m}$ and $V$ consists of $n-m$ columns which can be chosen from a set of size $|\mathcal{V}_Q|$. This gives $|\mathcal{T}_n| = |\mathcal{V}_Q|^{n-m}g_{n-m}$, because the matrix $A(\Gamma)$ is fixed. For any $Q$-respecting vector there are $\left(|\mathcal{V}_Q|-1\right)^{n-m}$ choices of $V$ that do not contain this $Q$-respecting vector as a column. The upper bound $|\mathcal{V}_Q|\left(|\mathcal{V}_Q|-1\right)^{n-m}$ on all $V$ that miss a $Q$-respecting vector is $o(1)$ of $|\mathcal{V}_Q|^{n-m}$. By Theorem~\ref{thm:asymmetric}, almost all of the graphs in $\mathcal{A}_{n-m}$ are asymmetric. Therefore, $\mathcal{S}_n$ contains almost all matrices in $\mathcal{T}_n$, so $|\mathcal{S}_n| =  |\mathcal{V}_Q|^{n-m}g_{n-m}(1-o(1))$.
\end{proof}

For the lower bound we need to control the isomorphisms even better. For this we need the following to hold:
\begin{enumerate}\addtocounter{enumi}{2}
    \item After switching there is no other $(Q,\Gamma)$-switching set in the graph.
    \item The set $[m]$ is the unique set on which you can perform a $(Q,\Gamma)$-switching in this graph.
\end{enumerate}

In terms of matrices we give the following definition.
\begin{definition}\label{def:snprime}
An element $A \in \mathcal{S}_n$ is \emph{exceptional} if the switched matrix $(Q\oplus I)^TA(Q \oplus I)$ has a $(Q,\Gamma)$-switching set that is not formed by the first $m$ rows and columns. An element $A \in \mathcal{S}_n$ has a \emph{unique switching} if $[m]$ is the only $(Q,\Gamma)$-switching set in the graph.  Let $\mathcal{S}_n'\subseteq \mathcal{S}_n$ be the subset of non-exceptional matrices that have a unique switching. 
\end{definition}

In this paragraph we explain the choices of the definitions for $\mathcal{S}_n$ and $\mathcal{S}_n'$. Lemma~\ref{lem:sizeS} shows that the set $\mathcal{T}_n$ already implies the upper bound $|\mathcal{V}_Q|^{n-m}g_{n-m}$. The problem with $\mathcal{T}_n$ is that it contains graphs which are unique up to isomorphism, such as many copies of the complete graph, and isomorphism classes for which many copies are in $\mathcal{T}_n$. To get the proper upper bound, we use Conditions $1$ and $2$ in combination with the assumption that $(Q,\Gamma)$-switching is distinguishing to lower bound the isomorphism classes. The proof of the lower bound needs more preparation. We need to find a subset of $\mathcal{T}_n$ that has nearly all graphs, where switching always gives a non-isomorphic graph and where we can control the sizes of the isomorphism classes in the set from above. This subset is $\mathcal{S}_n'$. Lemma~\ref{lem:sizeS'}, shows that $\mathcal{S}_n'$ only misses an asymptotically small fraction of the elements of $\mathcal{T}_n$. Conditions 1,2 and 4 are used to determine how large the isomorphism classes are in $\mathcal{S}_n'$, see Lemma~\ref{lem:isodoubless'}. It was shown in \cite{aidaiso} that there are graphs that are isomorphic to their switched counterpart but where the isomorphism cannot fix the switching set. This implies that something akin to Condition 4 is necessary for GM-switching. In the appendix, we show that this is not a special feature of GM-switching by giving an example of isomorphic graphs that are related by WQH-switching, but where no isomorphism fixes the switching set. 

Conditions 1, 2 and 3 are necessary to ensure that $(Q,\Gamma)$-switching on the graphs in $\mathcal{S}_n'$ always creates a non-isomorphic graph. These conditions suffice as long as the switching method is switching-distinguishing. 

\subsection{Upper bound}

\begin{lemma}\label{lem:upperbound}
    Given an orthogonal $m\times m$-matrix $Q$ and a distinguishing $(Q,\Gamma)$-switching method, then there are at most

     \[
    \frac{1}{|\aut_Q(\Gamma)|}|\mathcal{V}_Q|^{n-m}g_{n-m}(1+o(1))
    \]
    non-isomorphic graphs on $n$ vertices that have a cospectral mate through $(Q,\Gamma)$-switching.
\end{lemma}

\begin{proof}
    Recall that all graphs of order $n$ with a $(Q,\Gamma)$-switching are isomorphic to a graph in $\mathcal{T}_n$. The group $\aut_Q(\Gamma)$ acts on $\mathcal{T}_n$ by permuting the first $m$ rows and columns of the matrices. This action has trivial stabilizer for all elements in $\mathcal{S}_n$, because the switching method is distinguishing by assumption. Hence each isomorphism class of graphs that appears in $\mathcal{S}_n$ has at least $|\aut_Q(\Gamma)|$ representatives in $\mathcal{S}_n$. Thus the amount of non-isomorphic graphs represented in $\mathcal{T}_n$ is at most
    \begin{align*}
    \frac{1}{|\aut_Q(\Gamma)|}|\mathcal{S}_n| + |\mathcal{T}_n\setminus \mathcal{S}_n| &= \frac{1}{|\aut_Q(\Gamma)|}|\mathcal{V}_Q|^{n-m}g_{n-m}(1-o(1)) + |\mathcal{V}_Q|^{n-m}g_{n-m}o(1)\\
    &= \frac{1}{|\aut_Q(\Gamma)|}|\mathcal{V}_Q|^{n-m}g_{n-m}(1+o(1)).
    \end{align*}
   
    This is an upper bound on the non-isomorphic graphs that have a $(Q,\Gamma)$-switching tuple and hence also an upper bound on the graphs that have a cospectral mate through such a switching.    
\end{proof}

Note that the need for the switching method to be distinguishing in Lemma~\ref{lem:upperbound} is only relevant for the factor $\frac{1}{|\aut_Q(\Gamma)|}$. Therefore, we still get the upper bound $|\mathcal{V}_Q|^{n-m}g_{n-m}(1+o(1))$ without assuming that $(Q,\Gamma)$-switching is distinguishing.
\subsection{Lower bound}
This section contains four preliminary lemmas building to the lower bound that is derived in Lemma~\ref{lem:lowerbound}.

\begin{lemma}\label{lem:uniqueext}
    Given a tuple $X$ of $m$ distinct vertices, a vertex $x \in X$ and a graph $G$ that contains all vertices in $X$ except $x$. There is at most one way to choose the neighbors of $x$ in $G$ such that $G + \{x\}$ is a graph with a $(Q,\Gamma)$-switching on the tuple $X$.
\end{lemma}

\begin{proof}
    The edges in $X$ containing $x$ are determined, because we must have $A(X) = A(\Gamma)$. For a vertex $y$ outside $X$ we consider the corresponding $Q$-respecting vector $\v$, which becomes $Q^T\v$ after switching. If $\q_x$ is the column of $Q^T$ corresponding to $x$ and $v_x \in \{0,1\}$ is the entry of $\v$ corresponding to $x$, then $Q^T\v = \c + \q_xv_x$ where $\c$ is some vector independent of $x$. By definition of $Q$-switching, see Definition~\ref{def:qswitching}, any column of $Q$ is not an integral vector, so $\c$ and $\c+\q_x$ cannot both be integral column vectors. Thus there is at most one value of $v_x$ such that $Q^T\v$ is an integral vector and the adjacency between $y$ and $x$ is determined. This holds for all $y$ outside $x$. Therefore, all edges that contain $x$ are determined.
\end{proof}

Recall by Definition~\ref{def:snprime} that an element $A \in \mathcal{S}_n$ is exceptional if the switched matrix $(Q\oplus I)^TA(Q \oplus I)$ has a $(Q,\Gamma)$-switching set that is not formed by the first $m$ rows and columns, so a matrix is non-exceptional if there are no other switching sets after switching on one set. Similarly, we said that an element $A \in \mathcal{S}_n$ has a unique switching if $[m]$ is the only $(Q,\Gamma)$-switching set in the graph. We let $\mathcal{S}_n'\subseteq \mathcal{S}_n$ be the subset of non-exceptional matrices that have a unique switching. 

\begin{lemma}\label{lem:sizeS'}
    $|\mathcal{S}'_n| = |\mathcal{S}_n|(1-o(1))$.
\end{lemma}

\begin{proof}
    For any tuple $Y$ of $m$ distinct vertices that do not form the set $[m]$, we count the adjacency matrices $A\in \mathcal{S}_n'$ for which $Y$ is a $(Q,\Gamma)$-switching tuple of the switched graph $A'$. Pick a $y \in Y \setminus [m]$.  Let $B \in \mathcal{S}_{n-1}$ and $B' = (Q\oplus I)B(Q\oplus I)^T$. Lemma~\ref{lem:uniqueext} implies there is at most one way to choose the adjacencies of $y$ in $B'$ to ensure that $Y$ is a $(Q,\Gamma)$-switching tuple. Adding the $y$th row and column to $B'$ gives $A'$. From $A'$ one can construct $A = (Q \oplus I)A'(Q \oplus I)^T$. Therefore, each choice of $Y$ and $y \in Y\setminus [m]$ gives at most $1$ exceptional matrix in $\mathcal{S}_n$ for each matrix in $\mathcal{S}_{n-1}$. This gives at most $m\cdot m!\binom{n}{m}|\mathcal{S}_{n-1}|$ exceptional matrices in $\mathcal{S}_n$.
    
    For each $y \in Y\setminus [m]$ and $B \in \mathcal{S}_{n-1}$ there is also at most one way to add a $y$th row and column to $B$ to get a matrix in $\mathcal{S}_{n-1}$ that has $Y$ as a $(Q,\Gamma)$-switching tuple. This shows that $m \cdot m!\binom{n}{m}|\mathcal{S}_{n-1}|$ is also an upper bound on the number of matrices that do not have a unique switching. Lemma~\ref{lem:sizeS} implies that $|\mathcal{S}_n|$ grows more than exponentially in $n$, so $m\cdot m!\binom{n}{m}|\mathcal{S}_{n-1}| = o(1)|\mathcal{S}_n|$. We conclude that there are only $o(1)|\mathcal{S}_n|$ matrices that are not in $\mathcal{S}'_n$.
\end{proof}

\begin{lemma}\label{lem:isodoubless'}
    Each isomorphism class has at most $|\aut_Q(\Gamma)|$ elements in $\mathcal{S}_n'$.
\end{lemma}

Given two subsets $X$ and $Y$ of $[n]$, let $A[X, Y]$ denote the submatrix consisting of rows indexed by $X$ and columns indexed by $Y$. Write $A[X]$ for $A[X,X]$, such that $A[X]$ is the adjacency matrix of the induced subgraph $G[X]$.

\begin{proof}
    If $\phi: A \to B$ is an isomorphism between two graphs in $\mathcal{S}'_n$, then $\phi([m]) = [m]$, because the only $(Q,\Gamma)$-switching set of $B$ is $[m]$. Furthermore $A[[n]\setminus [m]]$ and $B[[n]\setminus [m]]$ are asymmetric elements of $\mathcal{A}_{n-m}$, so any isomorphism between them must be the identity. This implies that $\phi$ can only permute the elements of $[m]$. This needs to be an automorphism of $\Gamma$, because $A[[m]] = B[[m]] = A(\Gamma)$. Furthermore, the columns of $A[[m],[n]\setminus [m]]$ and $B[[m],[n]\setminus [m]]$ contain all $Q$-respecting vectors, so in fact $\phi$ needs to act like an element of $\aut_Q(\Gamma)$ on $[m]$. Therefore, there are at most $|\aut_Q(\Gamma)|$ graphs in $\mathcal{S}_n'$ isomorphic to $A$.
\end{proof}

Using Lemmas~\ref{lem:sizeS'} and \ref{lem:isodoubless'} we are ready to prove the lower bound.

\begin{lemma}\label{lem:lowerbound}
    For an orthogonal $m\times m$-matrix $Q$ and a $Q$-switching graph $\Gamma$ there are either $0$ or at least 
    \[
    \frac{1}{|\aut_Q(\Gamma)|}|\mathcal{V}_Q|^{n-m}g_{n-m}(1-o(1))
    \]
    graphs that have a cospectral mate through $(Q,\Gamma)$-switching.
\end{lemma}

\begin{proof} 
Assume that there is at least one graph with a cospectral mate by $(Q,\Gamma)$-switching. By Lemma~\ref{lem:isodoubless'} the set $\mathcal{S}'_n$ can be partitioned into isomorphism classes of size at most $|\aut_Q(\Gamma)|$. Together with Lemmas \ref{lem:sizeS} and \ref{lem:sizeS'} this implies that there are
    \begin{equation*}\label{eq:sizesprime}
        \frac{1}{|\aut_Q(\Gamma)|}|\mathcal{V}_Q|^{n-m}g_{n-m}(1-o(1))
    \end{equation*}
    
    non-isomorphic matrices in $\mathcal{S}'_n$. We only need to show that switching creates a non-isomorphic graph for all matrices in $\mathcal{S}'_n$. Let $A \in \mathcal{S}'_n$ be a matrix and $A'$ be the matrix after $(Q,\Gamma)$-switching. Suppose that there is an isomorphism $\phi: A \to A'$. Now $\phi([m]) = [m]$, because $A'$ does not have a $(Q,\Gamma)$-switching set besides $[m]$ since $A$ is non-exceptional. The switching acts on $[n]\setminus [m]$ as the identity, so $A[[n]\setminus [m]] = A'[[n]\setminus [m]]$, which is an asymmetric graph. Thus $\phi$, which fixes $[n]\setminus [m]$, must also act like the identity on this set, so $\phi$ can only permute the set $[m]$. Any column $\v$ of $A[[m], [n]\setminus [m]]$ gets sent to the same column in $A'$ by $\phi$, although potentially permuted. It must satisfy $\v_{\phi(i)} = (Q^T\v)_i$ for all $i$. This needs to hold for all $Q$-respecting vectors, because they all appear as one of the columns. Furthermore, $\phi$ restricts to an isomorphism between $A[[m]]$ and $A'[[m]]$. Such a $\phi$ cannot exist, because $(Q, \Gamma)$-switching is switching-distinguishing since it produces at least one cospectral pair. Hence no isomorphism between $A$ and $A'$ exists. This implies that all the matrices in $\mathcal{S}'_n$ have a cospectral mate through $(Q,\Gamma)$-switching, which proves the lemma.
\end{proof}

\section{Applications}\label{sec:consequences}

For all $(Q,\Gamma)$-switching methods one needs to take a few steps in order to apply the framework proposed in Theorem~\ref{thm:generalasymp}: 
\begin{itemize}
    \item Show that the method is distinguishing and switching-distinguishing.
    \item Determine $|\mathcal{V}_Q|$.
    \item Calculate $|\aut_Q(\Gamma)|$.
\end{itemize}

In case we want to aggregate the results per $Q$-switching we simply do the last step for each $Q$-switching graph $\Gamma$ and add the results, see Section~\ref{sec:addbounds} for more details.

In this section we perform these steps for all the switching methods presented in Section~\ref{sec:switching}. A summary of the outcome can be seen in Table \ref{tab:appliedformula}.

\begin{table}[H]
    \centering
    $\begin{array}{l| l}
         \text{switching} & \text{number of cospectral graphs}\\[0.1cm]
         \hline
         \mathrm{GM}_4 & 8^{n-4}g_{n-4}(\tfrac13+o(1))\\[0.1cm]
         \mathrm{WQH}_6 & 22^{n-6}g_{n-6}(\tfrac{43}{18}+o(1)) \\[0.1cm]
         \mathrm{GM}_6 & 22^{n-6}g_{n-6}(\tfrac{43}{180}+o(1))\\[0.1cm]
         \mathrm{AH}_6 & \tfrac23 16^{n-6}g_{n-6}(1+o(1))\\[0.1cm]
         \mathrm{Fano} & \frac{16}{7}16^{n-7}g_{n-7}(1+o(1))\\[0.1cm]
         \mathrm{WQH}_8 & \tfrac{20213}{576}72^{n-8}g_{n-8}(1+o(1)) 
         \\[0.1cm]
         \mathrm{GM}_8 & \tfrac{319}{280}72^{n-8}g_{n-8}(1+o(1))
            
    \end{array}$
    \caption{Asymptotic number of non-isomorphic graphs that have a cospectral mate per $Q$-switching.}
    \label{tab:appliedformula}
\end{table}
A switching method with a larger switching set enforces more conditions on the vertices in the graph, so intuitively we expect a method with a larger switching set to produce fewer cospectral graphs. The obtained expressions in Table     \ref{tab:appliedformula} clearly support this intuition.

We also note that each of the switching methods, except Fano switching, fall within the same family that we introduce in Section \ref{sec:infinitefam}. We can treat this family in a unified way to avoid repeating arguments for $\mathcal{V}_Q$ and being distinguishing for each individual switching method.

\subsection{Adding bounds for different switching methods}\label{sec:addbounds}

In this section we discuss when we can add the asymptotic bounds of Theorem~\ref{thm:generalasymp} for different switching methods. This allows us to aggregate the results per $Q$-switching, which is the normal way to distinguish the switching methods.

\begin{lemma}\label{lem:combine}
    If $(Q, \Gamma_1)$-switching and $(Q,\Gamma_2)$-switching are switching methods such that $\Gamma_1$ and $\Gamma_2$ are non-isomorphic graphs, then the number of graphs that have a cospectral mate through either $(Q, \Gamma_1)$-switching or $(Q,\Gamma_2)$-switching is
    \[
    \left(\frac{1}{\aut_Q(\Gamma_1)} + \frac{1}{\aut_Q(\Gamma_2)}\right)|\mathcal{V}_Q|^{n-m}g_{n-m}(1-o(1)).
    \]
\end{lemma}

\begin{proof}
    The given expression, which is $|\mathcal{S}_n'(Q,\Gamma_1)|+|\mathcal{S}_n'(Q,\Gamma_2)|$, is clearly an upper bound. Recall that $[m]$ is the $(Q,\Gamma_1)$-switching tuple for each of the graphs in $\mathcal{S}_n(Q,\Gamma_1)$. For any $A \in \mathcal{S}_n(Q,\Gamma_1)$ the $(Q,\Gamma_2)$-switching set $Y$ cannot be the set $[m]$, because $A[[m]]$ is the adjacency matrix of $\Gamma_1$, which is not isomorphic to $\Gamma_2$. Hence we can choose a vertex $y$ in $Y \setminus [m]$. By Lemma~\ref{lem:uniqueext} the $y$ column in $A$ is uniquely determined by the rest of the matrix, so the same argument as used in Lemma~\ref{lem:sizeS'} shows that the number of graphs in $\mathcal{S}_n(Q,\Gamma_1)$ that have a $(Q,\Gamma_2)$-switching set is $o(|\mathcal{S}_n(Q_1,\Gamma_1)|)$. Thus the number of isomorphism classes of graphs that are in $\mathcal{S}_n'(Q,\Gamma_1)$ and $\mathcal{S}_n'(Q,\Gamma_2)$ is negligible compared to the number of isomorphism classes in either set. Therefore, the sum $|\mathcal{S}_n'(Q,\Gamma_1)|+|\mathcal{S}_n'(Q,\Gamma_2)|$ is also a lower bound up to a $o(1)$ term.
\end{proof}

Lemma~\ref{lem:combine} is useful for adding switching methods for the same $Q$. By induction, this addition property holds for any finite amount of $Q$-switching graphs, so Lemma~\ref{lem:combine} can be used to combine results for the same $Q$ as we will see in the next sections.

\begin{remark}
    One can wonder if it is necessary to assume that $\Gamma_1$ and $\Gamma_2$ are non-isomorphic in Lemma~\ref{lem:combine}. This is not necessary as long as we can deduce that the $(Q,\Gamma_2)$-switching set cannot be the set $[m]$. However, it is a priori possible that $\Gamma_1$ and $\Gamma_2$ are distinct but isomorphic graphs that have the same $Q$-respecting vectors but such that no isomorphism fixes the $Q$-respecting vectors. In this case they have to be considered separate switching methods that cannot be added together. Note that this is exactly why we needed to talk about $(Q,\Gamma)$-switching tuples instead of sets. If such $\Gamma_1$ and $\Gamma_2$ exist, then the $(Q,\Gamma_1)$- and $(Q,\Gamma_2)$-switching applied to the same graph could produce non-isomorphic graphs. No example of $\Gamma_1$ and $\Gamma_2$ like this is known to the authors.
\end{remark}

The situation is also more tricky if one wishes to combine results for different $Q$. If a certain $(Q_1,\Gamma_1)$-switching is reducible, then by Definition~\ref{def:reducibility} there is a $(Q_2,\Gamma_2)$-switching that produces a cospectral mate for the same graph. Even for two indecomposable matrices $Q_1,Q_2$ of the same size there can still be significant overlap if $\mathcal{V}_{Q_1}$ and $\mathcal{V}_{Q_2}$ are the same up to a permutation that also sends $\Gamma_1$ to $\Gamma_2$. The authors are not aware of any such pair of matrices, but could not prove these do not exist.

\subsection{Infinite family}\label{sec:infinitefam}

All regular orthogonal matrices of level 2 are classified \cite{level2mats}. They consist of two sporadic cases and an infinite family of matrices given by 

\begin{equation*}
    \frac{1}{2}\left[
\arraycolsep=2.5pt\def\arraystretch{.9}\begin{array}{cccccc}
J       & O      & \cdots & \cdots & O      & Y      \\
Y       & J      & O      & \cdots & \cdots & O      \\
O       & Y      & J      & O      & \cdots & O      \\
        & \ddots & \ddots & \ddots & \ddots &        \\
O       & \cdots & O      & Y      & J      & O      \\
O       & \cdots & \cdots & O      & Y      & J
\end{array}
\right].
\end{equation*}
This infinite family can be extended to higher levels in such a way to include both GM- and WQH-switching. We define $Q(a,b,c)$ for integers $a,b$ and $c$ with $b\geq 2$, $c \in \{-1,1\}$ to be the $ab\times ab$-matrix (consisting of $b$ blocks per row) \begin{equation*} Q(a,b,c) :=
    \frac{1}{a}\left[
\arraycolsep=2.5pt\def\arraystretch{.9}\begin{array}{cccccc}
J_a       & O_a      & \cdots & \cdots & O_a      & cY_a      \\
cY_a       & J_a      & O_a      & \cdots & \cdots & O_a      \\
O_a       & cY_a      & J_a      & O_a      & \cdots & O_a      \\
        & \ddots & \ddots & \ddots & \ddots &        \\
O_a       & \cdots & O_a      & cY_a      & J_a      & O_a      \\
O_a       & \cdots & \cdots & O_a      & cY_a      & J_a
\end{array}
\right].
\end{equation*}  One can check that all these matrices are regular orthogonal, that $Q_{\text{\GM{2k}}} = Q(k,2,-1)$ and that $Q_{\text{\WQH{2k}}} = Q(k,2,1)$. Note that $Q(2,b,1)$ and $Q(2,b,-1)$ are equivalent matrices for any $b\geq 2$, so these give the same switching method. This family of orthogonal matrices includes all the matrices studied in \cite{switchingpaper} and \cite{MaoSimilar} except Fano switching.

The rest of this section prepares the proof of the values in Table~\ref{tab:appliedformula} for all $Q(a,b,c)$-switching methods. Lemma~\ref{lem:qrespinffam} and Lemma~\ref{lem:rqinffam} determine $|\mathcal{V}_Q|$. Whether the methods are switching-distinguishing and distinguishing is discussed in Lemma~\ref{lem:distinffam}.

\begin{lemma}\label{lem:qrespinffam}
    Let $Q=Q(a,b,c)$. For any subset $I \subseteq [m]$ and vector $\v \in \{0,1\}^m$, let $\sum v_I := \sum_{i \in I} v_i$. Let $C_i := \{(i-1)a+1,\dots, ia\}$ represent the set of indices corresponding to one block of $Q$. The $Q$-respecting vectors are exactly those vectors for which $\sum_{j \in C_i} v_j \equiv c\sum_{j \in C_{i+1}} v_j \pmod{a}$ for all $i \in \mathbb{Z}/b\mathbb{Z}$. 
\end{lemma}
\begin{proof}
    Consider a (0,1)-vector $\v$. This vector is $Q$-respecting exactly when each entry of $aQ^T\v$ is $0$ or $a$. Note that the sum of all positive entries in a row of $aQ^T$ is $2a-1$ and the sum of all negative entries is $1-a$. Thus any entry of $aQ^T\v$ is somewhere in the closed interval $[1-a,2a-1]$, so $\v$ is $Q$-respecting if and only if $aQ^T\v$ is zero modulo $a$. Any entry $(aQ^T\v)_j$ with $j \in \{1,\dots, a\}$ comes from $(J_a, cY_a,O,\dots,O)\v \equiv  (J_a, -cJ_a, O,\dots,O)\v \pmod{a}$. The vector $\v$ is $Q$-respecting when this is zero modulo $a$, so $\sum_{j \in C_1} v_j \equiv c\sum_{j \in C_2} v_j \pmod{a}$. The lemma follows from analogous statements for the other indices. 
\end{proof}

\begin{lemma}\label{lem:rqinffam}
Let $Q = Q(a,b,c)$, then
\[
     |\mathcal{V}_Q| = \begin{cases} 2^b + \binom{a}{a/2}^b, & \text{ if } (a,b,c) = (\text{even}, \text{odd}, -1),\\
    2^b, & \text{ if } (a,b,c) = (\text{odd}, \text{odd}, -1),\\
    2^b + \sum_{i = 1}^{a-1} \binom{a}{i}^b, & \text{ else.} \end{cases}
\]
\end{lemma}
\begin{proof}
    We consider the implications of Lemma~\ref{lem:qrespinffam} for the different cases in this lemma. Note that by induction, $\v$ is $Q$-respecting if and only if $\sum_{j \in C_i} v_j \equiv c^{k-i}\sum_{j \in C_k} v_j \pmod{a}$ for all $i,k$. 
\begin{description}
    \item[Case $c = 1$.] All $\sum_{j \in C_i} v_j$ are equal modulo $a$. Note that $0 \leq \sum_{j \in C_i} v_j \leq a$. There are $\binom{a}{t}$ ways for a block to have $\sum_{j \in C_i} v_j = t$. In case they all have the same value $t \in \{1,\dots, a-1\}$, there are $\binom{a}{t}^b$ options for $\v$. In case all $\sum_{j \in C_i} v_j$ are divisible by $a$ then each block can sum to $0$ or $a$, so all entries need to be the same. This gives $2^b$ options. Adding these options concludes this case.
    \item[Case $c = -1$ and $b$ even.] The blocks $C_i$ can be split into two classes depending on whether $i$ is even or odd. Lemma~\ref{lem:qrespinffam} then implies that all $\sum_{j \in C_i} v_j$ in the same class must have the same value  modulo $a$, which is minus the value of the opposite class. If $\sum_{j \in C_1} v_j = t \in \{1,\dots, a-1\}$, then $\sum_{j \in C_2} v_j = a-t$ and there are $\binom{a}{t} = \binom{a}{a-t}$ options for each of the $\sum_{j \in C_i} v_j$. This gives $\binom{a}{t}^b$ for each $t$. Again if $a \mid \sum_{j \in C_1} v_j$, then we get $2^b$ options in total.
    \item[Case $c = -1$ and $b$ odd.] In this case $$\sum_{j \in C_i} v_j \equiv \sum_{j \in C_{i+b}} v_j \equiv (-1)^b\sum_{j \in C_i} v_j \equiv -\sum_{j \in C_i} v_j \pmod{a},$$ so $a \mid 2\sum_{j \in C_i} v_j$ for each $i$ and $\sum_{j \in C_i} v_j \equiv \sum_{j \in C_k} v_j \pmod{a}$ for all $i,k$. If $a$ is odd, then $a \mid \sum_{j \in C_i} v_j$ for all $i$, which gives $2^b$ options. If $a$ is even, then one can also have $\sum_{j \in C_i} v_j = \frac{a}{2}$, which gives a further $\binom{a}{a/2}^b$ $Q$-respecting vectors.
\end{description}
\end{proof}

\begin{lemma}\label{lem:distinffam}
    Let $Q = Q(a,b,c)$ where \(a\geq2\) and $(a,b,c)$ is not $(\text{odd}, \text{odd}, -1)$. All $(Q,\Gamma)$-switchings are distinguishing and switching-distinguishing for all $Q$-switching graphs $\Gamma$.
\end{lemma}
\begin{proof}
    Let $\s\neq\t$ be any two vectors of length \(a\) with exactly $\left\lfloor\frac{a}{2}\right\rfloor$ ones and $\left\lceil\frac{a}{2}\right\rceil$ zeros such that \(\s\) has a one in the first position and \(\t\) has a zero in the first position (here we need \(a\geq2\)). Let $\o,\j$ be the all-zeros and all-ones vectors of length $a$ respectively. The following vectors are $Q$-respecting vectors whenever \((a,b,c)\neq(\text{odd}, \text{odd}, -1)\), see Lemma~\ref{lem:qrespinffam} and Lemma~\ref{lem:rqinffam}:
    \begin{align*}
        \x&=(\j,\o,\dots,\o)^T,\\
        \y&=\begin{cases}(\s,\s,\dots,\s,\s)^T&\text{if }c=1,\\(\s,\j-\s,\dots,\s,\j-\s)^T&\text{if }c=-1,\end{cases}\\
        \z&=\begin{cases}(\t,\s,\dots,\s,\s)^T&\text{if }c=1,\\(\t,\j-\s,\dots,\s,\j-\s)^T&\text{if }c=-1.\end{cases}
    \end{align*}
    The first entry is the only one that has a one in \(\y\) but not in \(\z\), so any permutation that sends $Q$-respecting vectors to themselves, needs to fix the first entry. For any other entry there exists a similar pair of vectors that ensure that this entry needs to be fixed (by symmetry of the condition in Lemma~\ref{lem:qrespinffam}). Therefore, the only permutation that maps $\v\to\v$ for every $\v\in\mathcal{V}_Q$ is the identity and thus any $(Q,\Gamma)$-switching is distinguishing. 

    Observe that \(Q^T\x=\x\), \(Q^T\y=\y\), and
    $$Q^T\z=\begin{cases}(\s,\s,\dots,\s,\t)^T&\text{if }c=1,\\(\s,\j-\s,\dots,\s,\j-\t)^T&\text{if }c=-1.\end{cases}$$
    The first entry has a one in \(\x\) and \(\y\) but not in \(\z\), while there are no entries that have a one in \(Q^T\x\) and \(Q^T\y\) but a zero in \(Q^T\z\). It follows that there is no permutation that sends $Q^T\v\to\v$ for each $\v\in \{\x,\y,\z\}$. Hence $Q$ is switching-distinguishing.
\end{proof}

We can now apply this general treatise of the infinite family to the particular cases of GM-, WQH- and AH-switching. We do this for all these switching methods up to size 8.

\subsubsection{\texorpdfstring{GM$_4$-switching}{GM4-switching}}

In \cite{enumeration} Haemers and Spence found $\tfrac{1}{24}n^3g_{n-1}(1-o(1))$ as a lower bound for the number of cospectral mates obtained by \GM4-switching. Theorem~\ref{thm:generalasymp} counts the graphs with a cospectral mate in a different way. Here we show that we arrive at the same bound.

\begin{corollary}
    The number of graphs on $n$ vertices with a cospectral mate through \GM4-switching is
\[
 \frac1{24}n^3g_{n-1}(1+o(1)).
\]
\end{corollary}

\begin{proof}
\GM4-switching corresponds to $Q(2,2,1)$-switching, so Lemma \ref{lem:distinffam} implies that all \GM4-switching methods are distinguishing and Lemma~\ref{lem:rqinffam} implies that $|\mathcal{V}_Q| = 8$. The matrix $Q_{\mathrm{GM}_4}$ is fully symmetric, so $|\aut_{Q_{\mathrm{GM}_4}}(\Gamma)| = |\aut(\Gamma)|$.  The \GM4-switching graphs are the four regular graphs on 4 vertices.
\begin{figure}[H]
        \centering
        \newcommand{\radius}{1.3}
\begin{subfigure}[b]{0.2\textwidth}
    \centering
    \begin{tikzpicture}[remember picture]
    \path (-\radius,-\radius) rectangle (\radius,\radius);
	\path[every node/.append style={circle, fill=black, minimum size=5pt, label distance=2pt, inner sep=0pt}]
    (45:1) node (0) {}
    (135:1) node (1) {}
    (225:1) node (2) {}
    (315:1) node[label={315:\(v\)}] (3) {};
    \end{tikzpicture}
\end{subfigure}
\hfill
\begin{subfigure}[b]{0.2\textwidth}
    \centering
    \begin{tikzpicture}
    \path (-\radius,-\radius) rectangle (\radius,\radius);
	\path[every node/.append style={circle, fill=black, minimum size=5pt, label distance=2pt, inner sep=0pt}]
    (45:1) node (0) {}
    (135:1) node (1) {}
    (225:1) node (2) {}
    (315:1) node[label={315:\(v\)}] (3) {};
    \draw (0) edge (1)
    (2) edge (3);
    \end{tikzpicture}
\end{subfigure}
\hfill
\begin{subfigure}[b]{0.2\textwidth}
    \centering
    \begin{tikzpicture}
    \path (-\radius,-\radius) rectangle (\radius,\radius);
	\path[every node/.append style={circle, fill=black, minimum size=5pt, label distance=2pt, inner sep=0pt}]
    (45:1) node (0) {}
    (135:1) node (1) {}
    (225:1) node (2) {}
    (315:1) node[label={315:\(v\)}] (3) {};
    \draw (0) edge (1)
    (1) edge (2)
    (2) edge (3)
    (3) edge (0);
    \end{tikzpicture}
\end{subfigure}
\hfill
\begin{subfigure}[b]{0.2\textwidth}
    \centering
    \begin{tikzpicture}
    \path (-\radius,-\radius) rectangle (\radius,\radius);
	\path[every node/.append style={circle, fill=black, minimum size=5pt, label distance=2pt, inner sep=0pt}]
    (45:1) node (0) {}
    (135:1) node (1) {}
    (225:1) node (2) {}
    (315:1) node[label={315:\(v\)}] (3) {};
    \draw (0) edge (1) edge (2)
    (1) edge (2) edge (3)
    (2) edge (3)
    (3) edge (0);
    \end{tikzpicture}
\end{subfigure}
        \caption{Regular graphs on 4 vertices.}
        \label{fig:boundpairs}
    \end{figure}
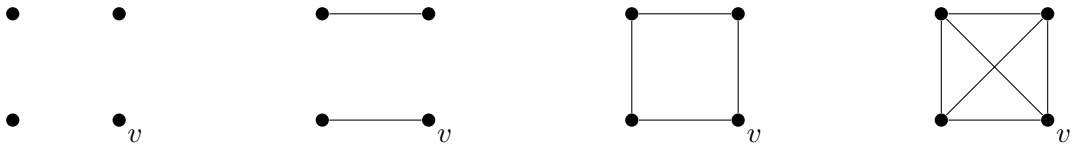
These have automorphism groups of sizes $24$, $8$, $8$ and $24$ respectively. Using Lemma~\ref{lem:combine} and Theorem~\ref{thm:generalasymp} we conclude that the number of graphs on $n$ vertices with a cospectral mate through \GM4-switching is
\begin{align*}
\left(\frac{1}{24}+ \frac{1}{8}+\frac18+\frac1{24}\right)8^{n-4}g_{n-4}(1+o(1)) &= \frac13\frac{2^{\binom{n-4}{2}+3(n-4)}}{(n-4)!}(1+o(1)) \\
&= \frac1{24}n^3g_{n-1}(1+o(1)).
\end{align*}
\end{proof}

\subsubsection{Six vertex switching}

The known switching methods on six vertices are \GM6, \WQH6 and \AH6. 

\begin{corollary}
    The number of graphs on $n$ vertices that have a cospectral mate through a switching method on $6$ vertices  is 
    \begin{itemize}
        \item $\frac{43}{18}22^{n-6}g_{n-6}(1+o(1))$ for \WQH6-switching.
        \item $\frac{43}{180}22^{n-6}g_{n-6}(1+o(1))$  for \GM6-switching.
        \item $\frac{2}{3}16^{n-6}g_{n-6}(1+o(1))$ for \AH6-switching.
    \end{itemize}
\end{corollary}

\begin{proof}
    The matrices for \WQH6, \GM6 and \AH6-switching are $Q(3,2,1)$, $Q(3,2,-1)$ and $Q(2,3,1)$ respectively, so Lemma~\ref{lem:distinffam} implies that the switching methods are all distinguishing and switching-distinguishing. Lemma~\ref{lem:rqinffam} says that $|\mathcal{V}_Q| = 22$ for \WQH6 and \GM6, and $|\mathcal{V}_Q| = 16$ for \AH6-switching. For each of these methods all $Q$-switching graphs are non-isomorphic. One can check by hand each of the $16, 8$ or $2$ graphs respectively.  
    
    \begin{figure}[H]
        \newcommand{\radius}{1.3}
\begin{subfigure}[b]{0.2\textwidth}
    \centering
    \begin{tikzpicture}[remember picture]
    \path (-\radius,-\radius) rectangle (\radius,\radius);
	\path[every node/.append style={circle, fill=black, minimum size=5pt, label distance=2pt, inner sep=0pt}]
    (0:0.5) node (0) {}
    (60:1) node (1) {}
    (120:1) node (2) {}
    (180:0.5) node (3) {}
    (240:1) node (4) {}
    (300:1) node (5) {};
    \end{tikzpicture}
    \caption{$720$}
\end{subfigure}
\hfill
\begin{subfigure}[b]{0.2\textwidth}
    \centering
    \begin{tikzpicture}
    \path (-\radius,-\radius) rectangle (\radius,\radius);
	\path[every node/.append style={circle, fill=black, minimum size=5pt, label distance=2pt, inner sep=0pt}]
    (0:0.5) node (0) {}
    (60:1) node (1) {}
    (120:1) node (2) {}
    (180:0.5) node (3) {}
    (240:1) node (4) {}
    (300:1) node (5) {};
    \draw (0) edge (3)
    (2) edge (1)
    (4) edge (5);
    \end{tikzpicture}
    \caption{$48$}
\end{subfigure}
\hfill
\begin{subfigure}[b]{0.2\textwidth}
    \centering
    \begin{tikzpicture}
    \path (-\radius,-\radius) rectangle (\radius,\radius);
	\path[every node/.append style={circle, fill=black, minimum size=5pt, label distance=2pt, inner sep=0pt}]
    ((0:0.5) node (0) {}
    (60:1) node (1) {}
    (120:1) node (2) {}
    (180:0.5) node (3) {}
    (240:1) node (4) {}
    (300:1) node (5) {};
    \draw (0) edge (1)
    (1) edge (2)
    (2) edge (3)
    (3) edge (4)
    (4) edge (5)
    (5) edge (0);
    \end{tikzpicture}
    \caption{$12$}
\end{subfigure}
\hfill
\begin{subfigure}[b]{0.2\textwidth}
    \centering
    \begin{tikzpicture}
    \path (-\radius,-\radius) rectangle (\radius,\radius);
	\path[every node/.append style={circle, fill=black, minimum size=5pt, label distance=2pt, inner sep=0pt}]
    (0:0.5) node (0) {}
    (60:1) node (1) {}
    (120:1) node (2) {}
    (180:0.5) node (3) {}
    (240:1) node (4) {}
    (300:1) node (5) {};
    \draw (0) edge (1) edge (2)
    (1) edge (2)
    (4) edge (3) edge (5)
    (5) edge (3);
    \end{tikzpicture}
    \caption{$72$}
\end{subfigure}
        \caption{All \GM6-switching graphs up to complementation and the size of their automorphism group.}
    \end{figure}
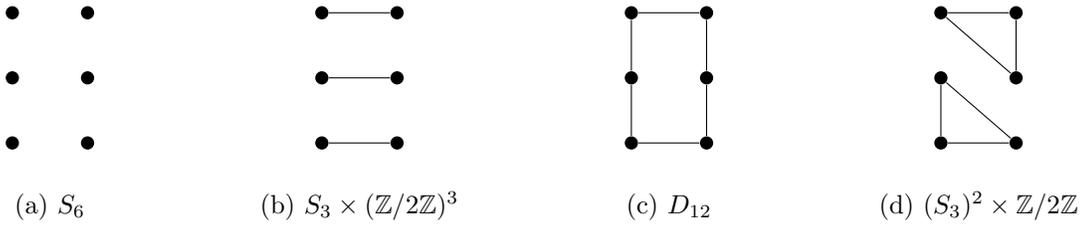
    
    These give sums of $|\aut_Q(\Gamma)|^{-1}$ equal to $\frac{43}{18}, \frac{43}{180}$ and $\frac23$.
    Therefore, Theorem~\ref{thm:generalasymp} gives the bounds given in the statement. 
\end{proof}

Note that $16^{n-6}g_{n-6} = n^6g_{n-2}(1+o(1))$. The proof for \GM4-switching as sketched by Haemers and Spence \cite{enumeration} constructed a graph with \GM4-switching from a graph on $n-1$ vertices. The equality suggests that a similar construction can be done for \AH6 from a graph on $n-2$ vertices. No such construction seems possible for \GM6 and \WQH6, as suggested by the fractions in the formulas. 

\subsubsection{Eight vertex switching}
After six vertex switching the next smallest switching methods from the infinite family are \GM8 and \WQH8. The corresponding matrices are $Q(4,2,1)$ and $Q(4,2,-1)$ respectively. Lemma~\ref{lem:distinffam} implies that the methods are distinguishing and switching-distinguishing and Lemma~\ref{lem:qrespinffam} implies that $|\mathcal{V}_Q| = 72$. We have enumerated the $22$ \GM8-switching graphs and $98$ \WQH8-switching graphs and checked their $|\aut_Q(\Gamma)|$  by computer, see the code at \url{https://github.com/nilsvandeberg/enumeration_switchingmethods.git}.

\begin{corollary}
    The number of graphs on $n$ vertices that have a cospectral mate through a switching method on $8$ vertices is 
    \begin{itemize}
        \item $\tfrac{20213}{576}72^{n-8}g_{n-8}(1+o(1))$ for \WQH8-switching.
        \item $\frac{319}{280}72^{n-8}g_{n-8}(1+o(1))$  for \GM8-switching.
    \end{itemize}
\end{corollary}

\subsection{Fano switching}  
In this section we perform the enumeration for an orthogonal matrix that is not part of the infinite family, the matrix $Q_\text{Fano}$. Fano switching was introduced in \cite{AHswitching}, and in \cite[~Section 6.2]{switchingpaper} the switching induced by this orthogonal matrix of order $7$ was given an interpretation in terms of the Fano plane (and thus was called \emph{Fano switching}). The two $Q_\text{Fano}$ switching graphs below were identified as the only graphs up to complementation that induced an irreducible switching.\\

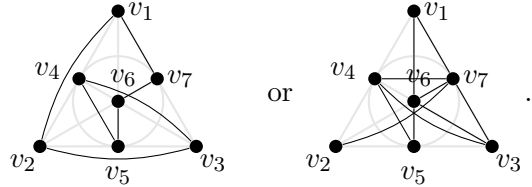
\begin{center}
 \(\newcommand{\radius}{1.2}
\begin{tikzpicture}[baseline=(O.base)]
    \path (0,0) node (O) {};
    \draw[thick,backline] (-30:\radius) -- coordinate (P1)
    (90:\radius) -- coordinate (P2)
    (210:\radius) -- coordinate (P3) cycle;
    \draw[thick,backline] (210:\radius) -- (P1) (-30:\radius) -- (P2) (90:\radius) -- (P3);
    \node[draw,thick,backline] at (O) [circle through=(P1)] {};
    \path[every node/.append style={circle, fill=black, minimum size=5pt, label distance=0pt, inner sep=0pt}]
    (O) node[label={[xshift=2pt]90:\(v_6\)}] (6) {}
    (P1) node[label={0:\(v_7\)}] (7) {}
    (P2) node[label={[label distance=3pt,yshift=3pt]180:\(v_4\)}] (4) {}
    (P3) node[label={[label distance=2pt]270:\(v_5\)}] (5) {}
    (-30:\radius) node[label={-40:\(v_3\)}] (3) {}
    (90:\radius) node[label={[label distance=0pt]0:\(v_1\)}] (1) {}
    (210:\radius) node[label={[label distance=0pt]220:\(v_2\)}] (2) {};
    \draw (1) edge[out=-135,in=75] (2)
    (2) edge[out=-15,in=195] (3)
    (3) edge[out=135,in=-15] (4)
    (4) edge (5)
    (5) edge (6)
    (6) edge (7)
    (7) edge (1);
\end{tikzpicture}\quad\text{ or } 
\begin{tikzpicture}[baseline=(O.base)]
    \path (0,0) node (O) {};
    \draw[thick,backline] (-30:\radius) -- coordinate (P1)
    (90:\radius) -- coordinate (P2)
    (210:\radius) -- coordinate (P3) cycle;
    \draw[thick,backline] (210:\radius) -- (P1) (-30:\radius) -- (P2) (90:\radius) -- (P3);
    \node[draw,thick,backline] at (O) [circle through=(P1)] {};
    \path[every node/.append style={circle, fill=black, minimum size=5pt, label distance=0pt, inner sep=0pt}]
    (O) node[label={[xshift=2pt]90:\(v_6\)}] (6) {}
    (P1) node[label={0:\(v_7\)}] (7) {}
    (P2) node[label={[label distance=3pt,yshift=3pt]180:\(v_4\)}] (4) {}
    (P3) node[label={[label distance=2pt]270:\(v_5\)}] (5) {}
    (-30:\radius) node[label={-40:\(v_3\)}] (3) {}
    (90:\radius) node[label={[label distance=0pt]0:\(v_1\)}] (1) {}
    (210:\radius) node[label={[label distance=0pt]220:\(v_2\)}] (2) {};
    \draw (4) edge[out=-45,in=165] (3) edge (5)
    (6) edge (1) edge (3) edge (5)
    (7) edge (1) edge[out=225,in=15] (2) edge (3) edge (4) edge (6);
\end{tikzpicture}\).
\end{center}

It can be checked that these graphs have $|\aut_Q(\Gamma)|$ equal to $7$ and $1$. 

\begin{lemma}\label{lem:Fanoapplicationmainmethod}
Fano switching is distinguishing and switching-distinguishing for all $Q_{\text{Fano}}$ switching graphs.
\end{lemma}

\begin{proof}
    The $16$ $Q_{Fano}$-respecting vectors are listed in \cite{AHswitching}. Represented as row vectors, these consist of $(0,0,0,0,0,0,0)$, $(1,1
    ,0,1,0,0,0)$ and their rotations and the complements of all these vectors. The first position is the only one that has a $1$ in each of the following vectors 
    \[
    (1,1,0,1,0,0,0), (1,0,1,0,0,0,1) \text{ and } (1,0,0,0,1,1,0), 
    \]
    so any permutation acting trivially on all $Q_\text{Fano}$-respecting vectors needs to fix the first entry. By rotating this triple of vectors we see that the permutation needs to be the identity, so $(Q_\text{Fano},\Gamma)$-switching is distinguishing for any $\Gamma$. After switching these vectors become
    \[
    (0,0,1,0,1,1,0), (0,1,0,1,1,0,0) \text{ and } (0,1,1,0,0,0,1).
    \]
    There is no position with a $1$ in all three of these vectors, so there is no permutation for which all $\v$ are sent to $Q_\text{Fano}^T\v$. Thus we conclude that $Q_\text{Fano}$ is switching-distinguishing for all $Q_{\text{Fano}}$ switching graphs.
\end{proof}

Lemma \ref{lem:Fanoapplicationmainmethod} shows that we can apply Theorem~\ref{thm:generalasymp} to Fano switching. 

\begin{corollary}
    There are 
    \[
    \frac{16}{7}16^{n-7}g_{n-7}(1+o(1))
    \]
    graphs on $n$ vertices that have a cospectral mate through Fano switching.
\end{corollary}

\section{Computer enumeration}\label{sec:computer}

Cospectral graphs have been enumerated previously with a complete enumeration up to 12 vertices \cite{cospectral12} as the largest effort. No such enumeration has been done for specific switching methods except by Haemers and Spence for GM-switching \cite{enumeration}. With the variety of different methods that have recently appeared in the literature \cite{WQHswitching,AHswitching,switchingpaper} this gives an incomplete picture. In this section we extend the work by Haemers and Spence to these newer switching methods. This gives a view on the amount of small cospectral graphs that can be obtained by means of each switching method. The code for this enumeration was written in SageMath \cite{sagemath} and ran on a computer cluster with a single core of an Intel Xeon Platinum 8260 CPU (2.4 GHz). The code can be found at \url{https://github.com/nilsvandeberg/enumeration_switchingmethods.git}. Here we give a sketch of our algorithm.\\

The code generates graphs using the block form described in the definition of $\mathcal{T}_n$. The set of $Q$-switching graphs $\mathcal{B}_Q$ is closed under complementation. For a fixed $Q$ and $n$ our approach takes one of each complement pair in $\mathcal{B}_Q$ and generates all adjacency matrices in $\mathcal{T}_n(Q,\Gamma)$. This is achieved by generating all matrices $V$ as a product of elements of $\mathcal{V}_Q$ and combining this with all graphs of order $n-|Q|$. This procedure potentially creates many isomorphic graphs, most of the time $|\aut_Q(\Gamma)|$ of each isomorphism class. In order to diminish the number of isomorphic graphs we deal with some of the $|\aut_Q(\Gamma)|$ symmetries by restricting one of the columns of $V$. In the next step the code checks for each of the generated graphs whether the switched matrix gives a non-isomorphic graph. This creates a list of graphs that have a cospectral mate through the fixed switching. Using canonical labels for each of these graphs, this list is reduced to a list where each isomorphism class of graphs occurs at most once. This approach avoids going through all graphs on $n$ vertices and uses the symmetries that are inherent in the switching method to reduce the number of graphs generated. The largest number of graphs created was 163,577,856 for \GM4-switching and order $10$. This instance took about 1000 hours to check. The numbers in the columns with a dagger in Tables \ref{tab:irredcospectral} and \ref{tab:gmcospectral} are taken from \cite{enumeration}. All the other columns are obtained by running our own code.

\begin{table}[H]
\resizebox{\textwidth}{!}{
    \begin{tabular}{c|| c| c ||c||c | c| c|| c|| c| c}
         order & NDGS$^\dagger$ & GM$^\dagger$ & \GM4 & \AH6 & \GM6 &  \WQH6 & Fano & \GM8 & \WQH8    \\
         \hline
         4 & 0 & 0 & 0 & - & - & - & - & - & - \\
         5 & 0 & 0 & 0 & - & - & - & - & - & -\\
         6 & 0 & 0 & 0 & 0  & 0 & 0 & -& - & - \\
         7 & 40 & 40 & 40 & 0 & 0 & 0 & 0 & - & -\\
         8 & 1166 & 1054 & 1030 & 48 & 48 & 94 & 24 & 0 & 0 \\
         9 & 43,811 & 38,258 & 37070 & 2242 & 2488 & 6212 & 502 & 96 & 96\\
         10 & 2,418,152 & 2,047,008 & 1,977,190 & 96,686 & 131,806 & 407,770  & 12,812 & 6096 & 11,498\\  
    \end{tabular}
    }
    \caption{Number of non-isomorphic graphs that have a cospectral mate per $Q$-switching.}
    \label{tab:irredcospectral}
\end{table}

\begin{table}[H]
\centering
\resizebox{0.85\textwidth}{!}{
    \begin{tabular}{c|| c| c ||c||c | c| c|| c|| c| c}
         order & NDGS$^\dagger$ & GM$^\dagger$ & \GM4 & \AH6 & \GM6 &  \WQH6 & Fano & \GM8 & \WQH8    \\
         \hline
         4 & 0 & 0 & 0 & - & - & - & - & - & - \\
         5 & 0 & 0 & 0 & - & - & - & - & - & -\\
         6 & 0 & 0 & 0 & 0  & 0 & 0 & -& - & - \\
         7 & 0.038 & 0.038 & 0.038 & 0 & 0 & 0 & 0 & - & -\\
         8 & 0.094 & 0.085 & 0.083 & 0.004 & 0.004 & 0.008 & 0.002 & 0 & 0 \\
         9 & 0.160 & 0.139 & 0.135 &  0.008 & 0.009 & 0.023 & 0.002 & 0.000 & 0.000 \\
         10 & 0.201 & 0.171 & 0.165 & 0.008 &  0.011 & 0.034 & 0.001 & 0.001 & 0.001
    \end{tabular}
    }
    \caption{Fraction, rounded to three decimals, of non-isomorphic graphs that have a cospectral mate per $Q$-switching.}
    \label{tab:fracirredcospectral}
\end{table}
 
  From Table~\ref{tab:irredcospectral}, it is clear that \GM4 produces by far the most cospectral mates among these switching methods. More generally we observe that in each row the numbers decrease with the size of the switching set. This matches the intuition, as mentioned at the start of Section~\ref{sec:consequences}. In Table~\ref{tab:compareformula} the formulas from Section~\ref{sec:consequences} are shown side-by-side with the computer results for $n = 10$.

\begin{table}[H]
\resizebox{\textwidth}{!}{
    \begin{tabular}{ l ||c||c | c| c|| c|| c| c}
         Method & \GM4 & \AH6 & \GM6 &  \WQH6 & Fano & \GM8 & \WQH8    \\
         \hline
         Enumeration & 1,977,190 & 96,686 & 131,806 & 407,770  & 12,812 & 6096 & 11,498\\  
         Formula (rounded to nearest integer) & 13,631,488 & 480,597 & 615,573 & 6,155,727 & 37,449 &  11,812 & 363,834\\
         Enumeration/Formula &0.145 &  0.201 & 0.214 &  0.066 & 0.342 & 0.516 &  0.032
    \end{tabular}
    }
    \caption{Comparison of asymptotic formula and exact enumeration for order 10.}
    \label{tab:compareformula}
\end{table}

 The fraction `Enumeration divided by Formula' should be converging to 1 for large $n$. The numbers in the last row of Table~\ref{tab:compareformula} are not close to $1$. We conclude that for $n = 10$  the $o(1)$ term is still significant, especially for the WQH-switching methods. For Fano and \GM8-switching the error is much smaller.

 Even though all considered switching methods are irreducible, there can still be an overlap in the graphs that have a cospectral mate through different switching methods. For example, the size of the intersections between the output lists for GM-switching are enumerated below. Through a quick application of inclusion-exclusion it can be checked that these numbers match the known enumeration results for GM-switching found in \cite{enumeration}.  It was mentioned in Section~\ref{sec:addbounds} that these intersections should be $o(1)$ compared to the enumeration of either switching method. Table~\ref{tab:gmcospectral} shows that in the range of $9$ vertices, the overlap is still significant.

\begin{table}[H]
\resizebox{\textwidth}{!}{
    \begin{tabular}{c||c |c|c | c| c| c| c|c}
         order & GM$^\dagger$ & \GM4 & \GM6 & \GM8 & \GM4 and \GM6 & \GM4 and \GM8 & \GM6 and \GM8 & \GM4, \GM6 and \GM8\\
         \hline
         7 & 40 & 40 & 0 & 0 & 0 & 0 & 0 & 0\\
         8 & 1054 & 1030 & 48 & 0 & 24 & 0 & 0 & 0\\
         9 & 38258 & 37070 & 2488 & 96 & 1336 & 40 & 38 & 18\\
         10 & 2047008 & 1977190 & 131806 & 6098 & 64550 & 2294 & 2388 & 1146
    \end{tabular}
    }
    \caption{Overlap of number of cospectral mates through different GM-switching methods.}
    \label{tab:gmcospectral}
\end{table}

\section{Concluding remarks}\label{sec:conclusion}

This paper provides a general framework to obtain the asymptotically tight number of cospectral mates produced by any switching method for the adjacency matrix. 

The proof of Theorem~\ref{thm:generalasymp} depends strongly on the independence of the entries of the adjacency matrix (the existence of an edge only influences two entries and has no implication for other entries). Thus this approach cannot immediately be applied to existing switching methods for other matrices, such as the Laplacian matrix. Indeed, for the Laplacian matrix, there exists a GM-like switching method where you enforce a stronger so-called GM$^\star$ condition, see \cite{enumeration}. In this case the row sum needs to be constant as well, so the columns of $V$ cannot be chosen independently. Fewer results are known about switching methods for the Laplacian, but the Laplacian appears to be more useful than the adjacency matrix in distinguishing graphs (see for example \cite{enumeration,AA2021}). This pattern could be supported by a version of Theorem~\ref{thm:generalasymp} for the Laplacian. Such a theorem would first require a more rigorous study of switching methods for the Laplacian matrix. A similar comment can be applied to other switching methods for other matrices associated to a graph, such as the distance matrix and the recently developed switching methods for it, see \cite{HEYSSE2017195, friesen2024cospectralconstructiongeneralizeddistance}.

In the case of simple graphs and the adjacency matrix, the formula given by Theorem~\ref{thm:generalasymp} can be applied to any new switching method. This immediately gives an indication of how the new switching method compares to existing switching methods in producing cospectral mates. More ambitiously, Theorem~\ref{thm:generalasymp} can also be used in the context of Haemers' conjecture. Indeed, a bound on the number of switching methods with matrices of level $2$ was used by Wang and Xu to find a family of graphs $\mathscr{H}_n$ of which almost all are determined by the generalized spectrum \cite{level2mats}. One could hope to use Theorem~\ref{thm:generalasymp} to extend this result on $\mathscr{H}_n$ to a larger family by considering higher levels of the orthogonal matrices. Later it was shown that all of $\mathscr{H}_n$ are in fact determined by the spectrum \cite{wangSimple}. This cannot be true for all graphs, but the restriction to the set of controllable graphs, which contains almost all graphs \cite{orourke2016touri}, implies that we only need to consider rational orthogonal matrices \cite{WANGsufficient} for Haemers' conjecture. Using Theorem~\ref{thm:generalasymp} for general bounds requires bounds on the number of $Q$-switching graphs and on $|\mathcal{V}_Q|$. The difficulty in this approach is dealing with all possible different $Q$. These are known and classified for level $2$ \cite{weighing}, but for any larger level a classification of regular orthogonal matrices does not exist. However, if there are some bounds on the number of (regular) orthogonal matrices of fixed sizes (for example derived from the generation of regular rational orthogonal matrices \cite{tang2024generationregularrationalorthogonal}), and the number of $Q$-switching graphs associated to them, then our result could help in tackling Haemers' conjecture through bounds for each switching method.

\subsection*{Acknowledgements}

Aida Abiad is supported by the Dutch Research Council (NWO) through the grants \linebreak VI.Vidi.213.085 and OCENW.KLEIN.475. Nils van de Berg is supported through the grant VI.Vidi.213.085. Robin Simoens is supported by the Research Foundation Flanders (FWO) through the grant 11PG724N.

\printbibliography[heading=bibintoc]

\newpage 
\appendix

\section*{Appendix: WQH-switching where isomorphism does not fix the switching set}

Here we give a construction of graphs where WQH-switching produces an isomorphic graph, but where no isomorphism fixes the switching set. This was done for GM-switching in \cite[Section 3]{aidaiso}. For $\ell \geq 2$ let $N$ be the $\ell \times \ell$-matrix that is zero except for one row of ones. Consider pairs of $k \times \ell$ matrices $V_1, V_2$ such that each column in $V_1$ contains $1 \leq t \leq k-1$ ones and the corresponding column in $V_2$ also has $t$ ones. For any such pair and any adjacency matrix $B$ of order $k$, consider the adjacency matrix
\[
A = \begin{pmatrix}
  O & J & J & V_1 & O & V_2 \\ J & O & O & V_2 & V_1 & O\\ J & O & O & O & V_2 & V_1  \\ V_1^T & V_2^T & O & B & N & N^T \\
    O & V_1^T & V_2^T &  N^T & B & N \\  V_2^T & O & V_1^T & N & N^T & B \\
\end{pmatrix}.
\]
After \WQH{2k} switching on the first two blocks this matrix becomes 
\[
A'= \begin{pmatrix}
  O & J & O & V_1 & O & V_2 \\ J & O & J & V_2 & V_1 & O\\ O & J & O & O & V_2 & V_1  \\ V_1^T & V_2^T & O & B & N & N^T \\
    O & V_1^T & V_2^T &  N^T & B & N \\  V_2^T & O & V_1^T & N & N^T & B \\
\end{pmatrix}.
\]
The obtained graph is isomorphic to the original graph through an isomorphism represented by acting on the blocks with the permutation $(321)(654)$.

Denote the blocks by $(C_i)_{i=1}^6$. Toward a contradiction, suppose that there is an isomorphism $\phi : A \to A'$ that fixes the switching set. There are no edges in $C_1$, so $\phi(C_1)$ must be fully contained in $C_1$ or in $C_2$. In both $A$ and $A'$ the set $C_3$ consists of the only vertices that are attached to all the vertices of one of the blocks and none of the other, so $\phi(C_3) = C_3$ and $\phi$ must swap $C_1$ and $C_2$. Each other block is uniquely determined by which of $C_1,C_2$ or $C_3$ it is not adjacent to. This implies that $C_4$ must be fixed and $C_5$ and $C_6$ must be swapped. The contradiction arises because there is no permutation within the blocks that turns $N$ into $N^T$.


\end{document}